\documentclass{amsart}
\usepackage{amsmath, amssymb, mathrsfs, verbatim, multirow}
\usepackage[all]{xy}

\newtheorem{Teo}{Theorem}[section]
\newtheorem{Prop}[Teo]{Proposition}

\newtheorem{Claim}[Teo]{Claim}

\theoremstyle{definition}
\newtheorem{Def}[Teo]{Definition}

\newtheorem{Obs}[Teo]{Remark}
\newtheorem{Que}[Teo]{Question}
\newtheorem{Exa}[Teo]{Example}

\newcommand{\Q}{\mathbb{Q}}

\newcommand{\N}{\mathbb{N}}

\newcommand{\F}{\mathbb{F}}

\newcommand{\Llr}{\Longleftrightarrow}
\newcommand{\lra}{\longrightarrow}
\newcommand{\Lra}{\Longrightarrow}

\newcommand{\SU}{\mbox{\rm supp}}

\begin{document}
\title{Extensions of a valuation from $K$ to $K[x]$}
\author{Josnei Novacoski}
\thanks{During part of the realization of this project the author was supported by a research grant from Funda\c c\~ao de Amparo \`a Pesquisa do Estado de S\~ao Paulo (process number 2017/17835-9).}

\begin{abstract}
In this paper we give an introduction on how one can extend a valuation from a field $K$ to the polynomial ring $K[x]$ in one variable over $K$. This follows a similar line as the one presented by the author in his talk at ALaNT 5. We will discuss the objects that have been introduced to describe such extensions. We will focus on key polynomials, pseudo-convergent sequences and minimal pairs. Key polynomials have been introduced and used by various authors in different ways. We discuss these works and the relation between them. We also discuss a recent version of key polynomials developed by Spivakovsky. This version provides some advantages, that will be discussed in this paper. For instance, it allows us to relate key polynomials, in an explicit way, to pseudo-convergent sequences and minimal pairs. This paper also provides examples that ilustrate these objects and their properties. Our main goal when studying key polynomials is to obtain more accurate results on the problem of local uniformization. This problem, which is still open in positive characteristic, was the main topic of the paper of the author and Spivakovsky in the proceedings of ALaNT 3.
\end{abstract}

\maketitle

\section{Introduction}
If $\nu_0$ is a valuation on a field $K$, what are the possible extensions $\nu$ of $\nu_0$ to $K[x]$? This question has been extensively studied and many objects have been introduced to describe such extensions. Three of the more relevant are key polynomials, pseudo-convergent sequences and minimal pairs. The main goal of this paper is to describe these objects and present the relation between them.

Throughout this paper, we will fix the following notations and assumptions:
\begin{equation}                           \label{sit}
\left\{\begin{array}{ll}
K & \mbox{is a field,}\\
\overline K & \mbox{is a fixed algebraic closure of }K,\\
K[x] & \mbox{is the ring of polynomials with one indeterminate over }K,\\
\nu & \mbox{is a valuation on }K[x],\\
\mu  &\mbox{is an extension of }\nu\mbox{ to }\overline{K}[x].\\
\end{array}\right.
\end{equation}

We start by defining key polynomials. These objects were introduced by MacLane in \cite{Mac} and refined by Vaqui\'e in \cite{Vaq}. The definition that we present here is slightly different and is due to Spivakovsky. The basic properties of Spivakovsky's key polynomials were developed in \cite{f} and will be sumarized in Section \ref{spivkegpplo}. In Section \ref{secpolchave} we will discuss the MacLane-Vaqui\'e key polynomials and in Section \ref{spivkegpplo} we discuss how they are related to Spivakovsky's key polynomials.

For a positive integer $b$, let
\[
\partial_b:=\frac{1}{b!}\frac{\partial^b}{\partial x^b}.
\]
For a polynomial $f\in K[x]$ let
\[
\epsilon(f)=\max_{b\in \N}\left\{\frac{\nu(f)-\nu(\partial_bf)}{b}\right\}.
\]
A monic polynomial $Q\in K[x]$ is said to be a (Spivakovsky's) \textbf{key polynomial} for $\nu$ if for every $f\in K[x]$,
\[
\epsilon(f)\geq \epsilon(Q)\Lra\deg(f)\geq\deg(Q).
\]

In \cite{Kap}, Kaplansky introduced the concept of pseudo-convergent sequences. For a valued $(K,\nu)$, a \textbf{pseudo-convergent sequence} is a well-ordered subset $\{a_{\rho}\}_{\rho<\lambda}$ of $K$, without last element, such that
\[
\nu(a_\sigma-a_\rho)<\nu(a_\tau-a_\sigma)\mbox{ for all }\rho<\sigma<\tau<\lambda.
\]
Let $R$ be a ring with $K\subseteq R$ and consider an extension of $\nu$ to $R$, which we call again $\nu$. An element $a\in R$ is said to be a limit of $\{a_\rho\}_{\rho<\lambda}\subseteq K$ if for every $\rho<\lambda$ we have $\nu(a-a_\rho)=\nu(a_{\rho+1}-a_\rho)$.

One of the main goals of \cite{f} is to compare key polynomials and pseudo-convergent sequences. These results are presented in Section \ref{compasraccai}.

Another theory that has been developed to study extensions of a given valuation to the the ring of polynomials in one variable is the theory of minimal pairs of definition of a valuation (see \cite{Kand1}). A \textbf{minimal pair for $\nu$} is a pair $(a,\delta)\in\overline K \times \mu(\overline K[x])$ such that for every $b\in\overline K$
\[
\mu(b-a)\geq \delta\Lra [K(b):K]\geq [K(a):K].
\]
If in addition,
\[
\mu(x-a)=\delta\geq \mu(x-b)
\]
for every $b\in \overline K$, then $(a,\delta)$ is called a \textbf{minimal pair of definition for $\nu$}.

The main goal of \cite{g} is to compare key polynomials and minimal pairs. These relations will be presented in Section \ref{compasraccai}.

For a valued field $(K,\nu)$ we denote by $K\nu$ the residue field and by $\nu K$ the value group of $\nu$, respectively. A valuation $\nu$ on $K[x]$ is called \textbf{valuation-algebraic} if $\nu(K(x))/\nu K$ is a torsion group and $K(x)\nu\mid K\nu$ is an algebraic extension. Otherwise, it is called \textbf{valuation-transcendental}. If $\nu$ is valuation-transcendental, then it is \textbf{residue-trascendental} if $K(x)\nu\mid K\nu$ is a transcendental extension and \textbf{value-transcendetal} if $\nu(K(x))/\nu K$ is not a torsion group.

Given two polynomials $f,q\in K[x]$ with $q$ monic, we call the \textbf{$q$-expansion of $f$} the expression
\[
f(x)=f_0(x)+f_1(x)q(x)+\ldots+f_n(x)q^n(x)
\]
where for each $i$, $0\leq i\leq n$, $f_i=0$ or $\deg(f_i)<\deg(q)$. For a polynomial $q(x)\in K[x]$, the \textbf{$q$-truncation of $\nu$} is defined as
\[
\nu_q(f):=\min_{0\leq i\leq n}\{\nu(f_iq^i)\}
\]
where $f=f_0+f_1q+\ldots+f_nq^n$ is the $q$-expansion of $f$.

We point out that the original definition of minimal pairs, presented in \cite{Kand1}, is slightly different than the one appearing here. The reason is because, with the original definition, one can prove that a valuation on $K[x]$ admits a pair of definition if and only if it is residue-transcendental. On the other hand, from the results in \cite{f}, one can prove that an extension admits a minimal pair of definition (as presented here) if and only if it is valuation-transcendental. Hence, with our definition we are considering all the valuations which are somehow simpler to handle. This result will follow from the following:
 
\begin{Teo}[Theorem 1.3 of \cite{g}]\label{Thmsobretruncvalutrans}
A valuation $\nu$ on $K[x]$ is valuation-transcendental if and only if there exists a polynomial $q\in K[x]$ such that $\nu=\nu_q$.
\end{Teo}

The theorem above can be seen as the version of Theorem 3.11 of \cite{Kuhbadplaces} for key polynomials and truncations. In Section \ref{spivkegpplo}, we describe a \textit{complete sequence of key polynomials for $\nu$}. If $\textbf{Q}$ is such sequence and $Q$ is a largest element for it, then $\nu=\nu_Q$. Hence, we conclude from Theorem \ref{Thmsobretruncvalutrans}, that if $\textbf{Q}$ has a last element, then $\nu$ is valuation-transcendental.

This paper is divided as follows. In Section \ref{secpolchave}, we describe the theory of MacLane-Vaqui\'e key polynomials. In Section \ref{spivkegpplo}, we describe some of the most important properties of Spivakovsky's key polynomials. Also in Section \ref{spivkegpplo}, we describe the relation of MacLane-Vaqui\'e and Spivakovsky's key polynomials. In Section \ref{pseudoconsequ}, we describe some of the main properties of pseudo-convergent sequences. Section \ref{compasraccai} is devoted to present the comparison between these three objects. In Section \ref{Exaampkes}, we present some examples that ilustrate the theory.

\section{Key polynomials}\label{secpolchave}

Take a commutative ring $R$ and an ordered abelian group $\Gamma$. Take $\infty$ to be an element not in $\Gamma$ and set $\Gamma_\infty$ to be $\Gamma\cup\{\infty\}$ with extensions of addition and order as usual.

\begin{Def}
A \textbf{valuation on $R$} is a map $\nu:R\lra \Gamma_\infty$ such that the following holds:
\begin{description}
\item[(V1)] $\nu(ab)=\nu(a)+\nu(b)$ for every $a,b\in R$,
\item[(V2)] $\nu(a+b)\geq \min\{\nu(a),\nu(b)\}$ for every $a,b\in R$,
\item[(V3)] $\nu(1)=0$ and $\nu(0)=\infty$.
\end{description}
\end{Def}

One can show that the \textbf{support of $\nu$}, defined by $\SU(\nu):=\{a\in R\mid \nu(a)=\infty\}$, is a prime ideal of $R$. Hence, if $R$ is a field, then \textbf{(V3)} is equivalent to
\[
\nu(x)=\infty\Llr x=0,
\]
which is the usual assumption for valuations defined on a field.

If $R=K[x]$, then valuations on $R$ describe all the valuations extending $\nu_0=\nu|_K$ to simple extensions $K(a)$ of $K$. Indeed, if $\SU(\nu)$ is the zero ideal, then $\nu$ extends in an obvious way to $K(x)$ where $x$ is a transcendental element.  If
\[
\SU(\nu)\neq (0),
\]
then there exists $p(x)\in K[x]$ monic and irreducible such that $\SU(\nu)= (p)$. Hence, $\nu$ defines a valuation on
\[
K[x]/(p)=K(a)
\]
for some element $a\in\overline K$ with minimal polynomial $p(x)$.

Let $\nu_0$ be a valuation of $K$ and $\nu$ a valuation of $K[x]$ extending $\nu_0$.
If $\gamma_1=\nu(x)$, we define
\[
\nu_1(a_0+a_1x+\ldots+a_rx^r)=\min\{\nu_0(a_i)+i\gamma_0\}.
\]
If $\nu=\nu_1$ we are done. If not, take a polynomial $\phi_1$ of smallest degree such that
\[
\gamma_1:=\nu(\phi_1)>\nu_1(\phi_1).
\]
For each $f\in K[x]$, write $f=f_0+f_1\phi_1+\ldots+f_r\phi_1^r$, with $\deg(f_i)<\deg(\phi_1)$ and define
\[
\nu_2(f)=\min\{\nu_1(f_i)+i\gamma_1\}.
\]
If $\nu=\nu_2$ we are done. Otherwise we continue the process.

\begin{Que}
Can we construct a ``sequence" of polynomials $\phi_i$ such that $\nu$ is the ``limit" of the maps $\nu_i$?
\end{Que}

Key polynomials were first introduced by MacLane in \cite{Mac}. In order to define Maclane key polynomials, we will need to define the graded algebra associated to a valuation. Let $R$ be a ring and $\nu$ a valuation on $R$. For every $\beta\in\nu R$, set
\[
P_\beta:=\{y\in R\mid\nu(y)\geq \beta\}\mbox{ and }P_\beta^+:=\{y\in R\mid\nu(y)> \beta\}.
\]
The graded algebra of $\nu$ is defined as
\[
\rm{gr}_\nu(R):=\bigoplus_{\beta\in\nu R}P_\beta/P_\beta^+.
\]
For an element $y\in R$ we denote by $\textrm{in}_\nu(y)$ the image of $y$ in $\rm{gr}_\nu(R)$, i.e.,
\[
\textrm{in}_\nu(y):=y+P_{\nu(y)}^+\in P_{\nu(y)}/P_{\nu(y)}^+\subset \textrm{gr}_\nu(R).
\]

Let $K$ be a field and let $\nu$ be valuation on $K[x]$, the polynomial ring in one variable over $K$. Given $f,g\in K[x]$, we say that \textbf{$f$ is $\nu$-equivalent to $g$} (and denote by $f\sim_\nu g$) if $\textrm{in}_\nu(f)=\textrm{in}_\nu(g)$. Moreover, we say that \textbf{$g$ $\nu$-divides $f$} (denote by $g\mid_\nu f$) if there exists $h\in K[x]$ such that $f\sim_\nu g\cdot h$.
\begin{Def}
A monic polynomial $\phi\in K[x]$ is a \textbf{Maclane-Vaqui\'e key polynomial for $\nu$} if it is $\nu$-irreducible (i.e., $\phi\mid_\nu f\cdot g \Lra \phi\mid_\nu f\mbox{ or }\phi\mid_\nu g$) and if for every $f\in K[x]$ we have
\[
\phi\mid_\nu f\Lra \deg(f)\geq \deg(\phi).
\]
\end{Def}
Let $\phi$ be a key polynomial for $\nu$, $\Gamma'$ be a group extension of $\nu(K[x])$ and $\gamma\in \Gamma'$ such that $\gamma>\nu(\phi)$. For every $f\in K[x]$, let
\[
f=f_0+f_1\phi+\ldots+f_n\phi^n
\]
be the $\phi$-expansion of $f$. Define the map
\[
\nu'(f):=\min_{0\leq i\leq n}\{\nu(f_i)+i\gamma\}.
\]
\begin{Teo}[Theorem 4.2 of \cite{Mac}]\label{Thmaugmentedval}
The map $\nu'$ is a valuation on $K[x]$.
\end{Teo}
\begin{Def}
The map $\nu'$ is called an \textbf{augmented valuation} and denoted by
\[
\nu':=[\nu;\nu'(\phi)=\gamma].
\]
\end{Def}

Given a valuation $\nu$ on $K$, a group $\Gamma'$ containing $\nu K$ and $\gamma\in \Gamma'$ we define the map
\[
\nu_\gamma(a_0+a_1x+\ldots+a_nx^n):=\min_{0\leq i\leq n}\{\nu(a_i)+i\gamma\}
\]
\begin{Teo}[Theorem 4.1 of \cite{Mac}]\label{Thmmonval}
The map $\nu_\gamma$ is a valuation on $K[x]$.
\end{Teo}
This valuation is called a \textbf{monomial valuation} and denoted by
\[
\nu':=[\nu;\nu'(x)=\gamma].
\]

Consider now the set $\mathcal V$ of all valuations on $K[x]$ (extending a fixed valuation $\nu_0$ on $K$). The theorems above give us an algorithm to build valuations on $K[x]$. Namely, take a group $\Gamma_1$ containing $\nu(K)$ and $\gamma_1\in \Gamma_1$. Set
\[
\nu_1:=[\nu_0;\nu_1(x)=\gamma_1].
\]
Now, let $\phi_1$ be a key polynomial for $\nu_1$, $\Gamma_2$ an extension of $\Gamma_1$ and $\gamma_2\in \Gamma_2$ with $\gamma_2>\nu_1(\phi_1)$. Set
\[
\nu_2:=[\nu_1;\nu_2(\phi_1)=\gamma_2].
\]
Proceding interatively, we build groups
\[
\nu(K)\subseteq\Gamma_1\subseteq\Gamma_2\subseteq\ldots\subseteq\Gamma_n\subseteq\ldots,
\]
valuations
\[
\nu_1,\nu_2,\ldots,\nu_n,\ldots\in\mathcal V,
\]
polynomials
\[
\phi_1,\ldots,\phi_n,\ldots\in K[x]
\]
and $\gamma_i\in\Gamma_i$, $i\in\N$ such that $\phi_{i+1}$ is a key polynomial for $\nu_i$ and
\[
\nu_{i+1}=[\nu_i;\nu_{i+1}(\phi_{i})=\gamma_{i+1}].
\]

Assume that we have constructed an infinite sequence as above. Let $\Gamma_\infty$ be a group, such that $\cup \Gamma_i\subseteq \Gamma_\infty$ such that every non-empty subset of $\Gamma_\infty$ admits a supremum. For every $f\in K[x]$, we define
\[
\nu_\infty(f):=\sup\{\nu_i(f)\}.
\]
\begin{Teo}[Theorem 6.2 of \cite{Mac}]
The map $\nu_\infty$ is a valuation of $K[x]$.
\end{Teo}
The valuation constructed in the theorem above is called a \textbf{limit valuation} (and we denote $\nu_\infty=\lim\nu_i$.

Consider now the subset $\mathcal V^c$ of $\mathcal V$ consisting of monomial, augmented and limit valuations (extending $\nu_0$).
\begin{Que}\label{QuestMaclane}
Is it true that $\mathcal V^c=\mathcal V$? In other words, given any valuation $\nu\in\mathcal V$, does there exist a sequence of valuation $\nu_1,\nu_2,\ldots, \nu_n,\ldots$ such that $\nu=\nu_i$ for some $i$ or $\nu=\lim\nu_i$?
\end{Que}
Let $\nu$ be any valuation on $K[x]$. Set
\[
\nu_0=[\nu;\nu_0(x)=\nu(x)].
\]
If $\nu=\nu_0$, then $\nu\in \mathcal V^c$. If not, then take $\phi_1\in K[x]$, monic and of smallest degree among polynomials $f$ satisfying $\nu_1(f)<\nu(f)$. One can prove that $\phi_1$ is a key polynomial for $\nu_1$. Consider then the valuation
\[
\nu_2=[\nu_1;\nu_2(\phi_1)=\nu(\phi_1)].
\]
If $\nu_2=\nu$, then $\nu\in \mathcal V^c$. If not, we choose $\phi_2\in K[x]$ monic and of smallest degree among polynomials satisfying $\nu_2(f)<\nu(f)$. Again, one can prove that $\phi_2$ is a key polynomial for $\nu_2$ and consider
\[
\nu_2=[\nu_1;\nu_2(\phi_2)=\nu(\phi_2)].
\]
We proceed iteratively until we find a valuation $\nu_n$ with $\nu_n=\nu$, or constructing an infinite sequence $\{\nu_i\}_{i\in\N}$ such that $\nu_i\neq\nu$ and $\nu_{i+1}$ is an augmented valuation of $\nu_i$. We have the following:
\begin{Teo}[Theorem 8.1 of \cite{Mac}]\label{Maclkpthem}
If $\nu_0$ is a discrete valuation of $K$, and the infinite sequence above has been constructed, then $\nu=\lim\nu_i$. In particular, if $\nu_0$ is a discrete valuation, then $\mathcal V^c=\mathcal V$.
\end{Teo}
If $\nu_0$ is not discrete, then $\mathcal V^c$ does not have to be equal to $\mathcal V$ (as it will be shown in Section \ref{Exaampkes}). This happens because we might need a sequence of key polynomials of order type greater than $\omega$. In order to find a sequence of ``augumented" valuations for a given valuation, Vaqui\'e introduced ``limit key valuations" (associated to a \textit{limit key polynomial}).

A family $\{\nu_\alpha\}_{\alpha\in A}$ of valuations of $K[x]$, indexed by a totally ordered set $A$, is called a \textbf{family of augmented iterated valuations} if for all $\alpha$ in $A$, except $\alpha$
the smallest element of $A$, there exists $\theta$ in $A$, $\theta<\alpha$, such that the valuation $\nu_\alpha$ is an augmented valuation of the form $\nu_\alpha=[\nu_\theta;\nu_\alpha(\phi_\alpha)=\gamma_\alpha]$, and if we have the following properties:
\begin{itemize}
\item If $\alpha$ admits an immediate predecessor in $A$, $\theta$ is that predecessor, and in the case when $\theta$ is not the smallest element of $A$, the polynomials $\phi_\alpha$ and $\phi_\theta$ are not $\nu_\theta$-equivalent and satisfy $\deg(\phi_\theta)\leq \deg(\phi_\alpha)$;
\item if $\alpha$ does not have an immediate predecessor in $A$, for all $\beta$ in $A$ such that $\theta<\beta<\alpha$, the valuations $\nu_\beta$ and $\nu_\alpha$ are equal to the augmented valuations
\[
\nu_\beta=[\nu_\theta;\nu_\beta(\phi_\beta)=\gamma_\beta]\mbox{ and }\nu_\alpha=[\nu_\beta;\nu_\alpha(\phi_\alpha)=\gamma_\alpha],
\]
respectively, and the polynomials $\phi_\alpha$ and $\phi_\beta$ have the same degree.
\end{itemize}

For $f,g\in K[x]$, we say that $f$ \textbf{$A$-divides $g$} ($f\mid_A g$) if there exists $\alpha_0\in A$ such that  $f\mid_{\nu_\alpha}g$ for every $\alpha\in A$ with $\alpha>\alpha_0$. A polynomial $\phi$ is said to be \textbf{$A$-minimal} if for any polynomial $f\in K[x]$ if $\phi\mid_A f$, then $\deg(\phi)\leq \deg(f)$. Also, we say that $\phi$ is \textbf{$A$-irreducible} if for every $f,g \in K[x]$, if $\phi\mid_A f\cdot g$, then $\phi\mid_A f$ or $\phi\mid_A g$.
\begin{Def}
A monic polynomial $\phi$ of $K[x]$ is said to be a \textbf{Maclane-Vaqui\'e limit key polynomial} for the family $\{\nu_\alpha\}_{\alpha\in A}$ if it is $A$-minimal and $A$-irreducible.
\end{Def}

Let $\{\nu_\alpha\}_{\alpha\in A}$ be a family of iterated valuations of $K[x]$ and, for each $\alpha\in A$, denote the value group of $\nu_\alpha$ by $\Gamma_{\nu_\alpha}$. Then
\[
\Gamma_A:=\bigcup_{\alpha\in A} \Gamma_{\nu_\alpha}
\]
is a totally ordered abelian group. For a polynomial $f\in K[x]$, the family $\{\nu_\alpha\}_{\alpha\in A}$ is said to be convergent for $f$ if $\{\nu_\alpha(f)\}_{\alpha\in A}$ admits a majorant in $\Gamma_A$.

\begin{Teo}[Th\'eor\`eme 2.4 of \cite{Vaq}]\label{vquithrm}
Let $\nu$ be a valuation of $K[x]$ extending a valuation $\nu_0$ of $K$. Then, there exists a family of iterated valuations $\{\nu_\alpha\}_{\alpha\in A}$ of $K[x]$, convergent for every $f\in K[x]$, such that
\[
\nu(f)=\lim_{\alpha\in A}\{\nu_\alpha(f)\}:=\sup_{\alpha\in A}\{\nu_\alpha(f)\}.
\]
\end{Teo}

\begin{Obs}
Theorem \ref{vquithrm} is a generalization of Theorem \ref{Maclkpthem}. The difference is that, if $\nu$ is not discrete, we might need a sequence of key polynomials with order type greater than $\omega$.
\end{Obs}
\section{Spivakovsky's key polynomials}\label{spivkegpplo}
We start this section by presenting a characterization of $\epsilon(f)$ is terms of the fixed extension $\mu$ of $\nu$ to $\overline K[x]$. For a monic polynomial $f\in K[x]$, we define
\[
\delta(f):=\max\{\nu(x-a)\mid a\mbox{ is a root of }f\}.
\]
\begin{Exa}
Let $f(x)=(x-a_1)(x-a_2)(x-a_3)$. Then
\begin{displaymath}
\begin{array}{rcl}
\partial_1(f)&=&(x-a_1)(x-a_2)+(x-a_1)(x-a_3)+(x-a_2)(x-a_3)\\
\partial_2(f)&=&(x-a_1)+(x-a_2)+(x-a_3)\\
\partial_3(f)&=&1.
\end{array}
\end{displaymath}
\textbf{(i)} Assume that $\mu(x-a_i)=i$, for $i=1,2,3$, then
\[
\nu(f)=6,\ \nu(\partial_1f)=3,\ \nu(\partial_2f)=1\mbox{ and }\nu(\partial_3f)=0,
\]
and hence
\begin{displaymath}
\begin{array}{rcl}
\epsilon(f)&=&\max\left\{\frac{\nu(f)-\nu(\partial_1f)}{1},\frac{\nu(f)-\nu(\partial_2f)}{2},\frac{\nu(f)-\nu(\partial_3f)}{3}\right\}\\ \\
&=&\max\left\{3,\frac{5}{2},2\right\}=3=\delta(f).
\end{array}
\end{displaymath}
\textbf{(ii)} Assume that $\mu(x-a_1)=1$ and $\mu(x-a_2)=\mu(x-a_3)=2$, then
\[
\nu(f)=5,\ \nu(\partial_1f)\geq 3,\ \nu(\partial_2f)=1\mbox{ and }\nu(\partial_3f)=0,
\]
and hence
\begin{displaymath}
\begin{array}{rcl}
\epsilon(f)&=&\max\left\{\frac{\nu(f)-\nu(\partial_1f)}{1},\frac{\nu(f)-\nu(\partial_2f)}{2},\frac{\nu(f)-\nu(\partial_3f)}{3}\right\}\\ \\
&=&\max\left\{2,2,\frac{5}{3}\right\}=2=\delta(f).
\end{array}
\end{displaymath}
\end{Exa}

The examples above can be generalized to prove the following.
\begin{Prop}[Proposition 3.1 of \cite{g}]\label{Profmagica}
Let $f\in K[x]$ be a monic polynomial. Then $\delta(f)=\epsilon(f)$.
\end{Prop}
In particular, $\delta(f)$ does not depend on the choice of the extension $\mu$ of $\nu$ to $\overline K[x]$.

Let $q\in K[x]$ be any polynomial. Then $\nu_q$ does not need to be a valuation (Example 2.5 of \cite{f}). The first important property of key polynomials is the following.

\begin{Prop}[Proposition 2.6 of \cite{f}]
If $Q$ is a key polynomial, then $\nu_Q$ is a valuation.
\end{Prop}
We observe that the converse of the above Proposition is not true, i.e., there exists a valuation $\nu$ on $K[x]$ and polynomial $q\in K[x]$ such that $\nu_q$ is a valuation, but $q$ is not a key polynomial (Corollary 2.4 of \cite{g}).

For a key polynomial $Q\in K[x]$, let
\[
\alpha(Q):=\min\{\deg(f)\mid \nu_Q(f)< \nu(f)\}, \mbox{ and }
\]
\[
\Psi(Q):=\{f\in K[x]\mid f\mbox{ is monic},\nu_Q(f)< \nu(f)\mbox{ and }\alpha(Q)=\deg (f)\}.
\]

\begin{Teo}[Theorem 2.12 of \cite{f}]\label{definofkeypol}
A monic polynomial $Q$ is a key polynomial if and only if there exists a key polynomial $Q_-\in K[x]$ such that either $Q\in \Psi(Q_-)$ or the following conditions are satisfied:
\begin{description}
\item[(K1)] $\alpha(Q_-)=\deg (Q_-)$
\item[(K2)] the set $\{\nu(Q')\mid Q'\in\Psi(Q_-)\}$ does not contain a maximal element
\item[(K3)] $\nu_{Q'}(Q)<\nu(Q)$ for every $Q'\in \Psi(Q_-)$
\item[(K4)] $Q$ has the smallest degree among polynomials satisfying \textbf{(K3)}.
\end{description}
\end{Teo}

\begin{Def}
A key polynomial $Q$ is called a (Spivakovsky's) \textbf{limit key polynomial} if the conditions \textbf{(K1) - (K4)} of the theorem above are satisfied.
\end{Def}
For a set $\textbf{Q}\subseteq K[x]$ we denote by $\N^\textbf{Q}$ the set of mappings $\lambda:\textbf{Q}\lra \N$ such that $\lambda(q)=0$ for all, but finitely many $q\in\textbf{Q}$. For $\lambda\in \N^\textbf{Q}$ we denote
\[
\textbf{Q}^\lambda:=\prod_{q\in\textbf{Q}}q^{\lambda(q)}\in K[x].
\]
\begin{Def}
A set $\textbf{Q}\subseteq K[x]$ is called a complete set for $\nu$ if for every $p\in K[x]$ there exists $q\in \textbf{Q}$ such that
\begin{equation}\label{eqabotdegpol}
\deg(q)\leq\deg(p)\mbox{ and }\nu(p)=\nu_q(p).
\end{equation}
\end{Def}

\begin{Prop}\label{propsobrecomp}
If $\textbf Q\subseteq K[x]$ is a complete set for $\nu$, then for every $p\in K[x]$ there exist $a_1,\ldots,a_r\in K$ and $\lambda_1,\ldots, \lambda_r\in \N^\textbf{Q}$, such that
\[
p=\sum_{i=1}^ra_i\textbf{Q}^{\lambda_i}\mbox{ with }\nu\left(a_i\textbf{Q}^{\lambda_i}\right)\geq \nu(p),\mbox{ for every }i, 1\leq i\leq r,
\]
and the elements $Q\in\textbf{Q}$ appearing in the decomposition of $p$ (i.e., for which $\lambda_i(Q)\neq 0$ for some $i$, $1\leq i\leq r$) have degree smaller or equal than $\deg(p)$. In particular, for every $\beta\in\nu(K[x])$, the additive group $P_\beta$ is generated by the elements $a \textbf Q^\lambda\in P_\beta$ where $a\in K$ and $\lambda\in\N^\textbf{Q}$.
\end{Prop}
\begin{Obs}
The latter condition on the proposition above appears as the definition of \textit{generating sequence} in various works.
\end{Obs}
\begin{proof}[Proof of Proposition \ref{propsobrecomp}]
We will prove our result by induction on the degree of $p$. If $\deg(p)=1$, then $p=x-a$ for some $a\in K$. By our assumption, there exists $x-b\in \textbf{Q}$ such that
\[
\beta:=\nu(x-a)=\nu_{x-b}(x-a)=\min\{\nu(x-b),\nu(b-a)\}.
\]
This implies that $\nu(x-b)\geq \beta$, $\nu(b-a)\geq \beta$ and that $p=(x-b)+(b-a)$, which is what we wanted to prove.

Assume now that for $k\in\N$, for every $p\in K[x]$ of $\deg(p)<k$ our result is satisfied. Let $p$ be a polynomial of degree $k$. Since $\textbf{Q}$ is a complete set for $\nu$, there exists $q\in \textbf{Q}$ such that $\deg(q)\leq \deg(p)$ and $\nu_q(p)=\nu(p)$. Let
\[
p=p_0+p_1q+\ldots+p_sq^s
\]
be the $q$-expansion of $p$. Since $\deg(q)\leq \deg(p)$, we have $\deg(p_i)<\deg(p)=k$ for every $i$, $1\leq i\leq s$. By the induction hypothesis, there exist
\[
a_{11},\ldots, a_{1r_1},\ldots,a_{s1},\ldots,a_{sr_s}\in K\mbox{ and }\lambda_{11},\ldots, \lambda_{1r_1},\ldots,\lambda_{s1},\ldots,\lambda_{sr_s}\in \N^\textbf{Q},
\]
such that for every $i$, $0\leq i\leq s$,
\[
p_i=\sum_{j=1}^{r_i}a_{ij}\textbf{Q}^{\lambda_{ij}}\mbox{ with }\nu\left(a_{ij}\textbf{Q}^{\lambda_{ij}}\right)\geq \nu(p_i)\mbox{ for every }j, 1\leq j\leq r_i,
\]
and $\deg(Q)\leq \deg(p_i)\leq \deg(p)$ for every polynomial appearing in the decompostion of $p_i$.
This implies that
\[
p=\sum_{i=0}^s\left(\sum_{j=1}^{r_i}a_{ij}\textbf{Q}^{\lambda_{ij}}\right)q^i=\sum_{0\leq i\leq s, 1\leq j\leq r_i}a_{ij}\textbf Q^{\lambda_{ij}'},
\]
where
\begin{displaymath}
\lambda'_{ij}(q')=\left\{
\begin{array}{ll}
\lambda_{ij}(q')+i&\mbox{ if }q'=q\\
\lambda_{ij}(q')&\mbox{ if }q'\neq q
\end{array}
\right..
\end{displaymath}
Moreover, since $\nu_q(p)=\displaystyle\min_{0\leq i\leq s}\{\nu(p_iq^i)\}=\nu(p)$ and
\[
\nu\left(a_{ij}\textbf{Q}^{\lambda_{ij}}\right)\geq \nu(p_i),\mbox{ for every }i, 0\leq i\leq n\mbox{ and }j, 1\leq j\leq r_i,
\]
we have
\[
\nu(p)\leq \nu(p_i)+i\nu(q)\leq \nu\left(a_{ij}\textbf{Q}^{\lambda_{ij}}\right)+i\nu(q)=\nu\left(a_{ij}\textbf{Q}^{\lambda'_{ij}}\right),
\]
for every $i$, $0\leq i\leq s$ and $j$, $1\leq j\leq r_i$, which is what we wanted to prove.
\end{proof}

The next result gives us a converse for Proposition \ref{propsobrecomp}.

\begin{Prop}
Assume that $\textbf{Q}$ is a subset of $K[x]$ with the following properties:
\begin{itemize}
\item $\nu_Q$ is a valuation for every $Q\in\textbf{Q}$;
\item for every finite subset $\mathcal F\subseteq \textbf{Q}$, there exists $Q\in \mathcal F$ such that $\nu_Q(Q')=\nu(Q')$ for every $Q'\in\mathcal{F}$;
\item for every $p\in K[x]$ there exist $a_1,\ldots,a_r\in K$ and $\lambda_1,\ldots,\lambda_r\in\N^\textbf{Q}$ such that 
\[
p=\sum_{i=1}^ra_i\textbf{Q}^{\lambda_i}\mbox{ with }\nu\left(a_i\textbf{Q}^{\lambda_i}\right)\geq \nu(p),\mbox{ for every }i, 1\leq i\leq r,
\]
and $\deg(Q)\leq\deg(p)$ for every $Q\in \textbf{Q}$ for which $\lambda_i(Q)\neq 0$ for some $i$, $1\leq i\leq r$.
\end{itemize}
Then $\textbf Q$ is a complete set for $\nu$.
\end{Prop}
\begin{proof}
Take any polynomial $p\in K[x]$ and let $\beta:=\nu(p)$. Then, there exist $a_1,\ldots,a_r\in K$ and $\lambda_1,\ldots,\lambda_r\in\N^\textbf{Q}$ such that 
\[
p=\sum_{i=1}^ra_i\textbf{Q}^{\lambda_i}\mbox{ with }\nu\left(a_i\textbf{Q}^{\lambda_i}\right)\geq \beta,\mbox{ for every }i, 1\leq i\leq r,
\]
and $\deg(Q)\leq\deg(p)$ for every $Q\in \textbf{Q}$ for which $\lambda_i(Q)\neq 0$ for some $i$, $1\leq i\leq r$.
Let
\[
\mathcal F:=\{Q\in \textbf{Q}\mid \lambda_i(Q)\neq 0\mbox{ for some }i,1\leq i\leq n\}.
\]
Since $\mathcal F$ is finite, there exists $Q\in\mathcal F$ such that $\nu_Q(Q')=\nu(Q')$ for every $Q'\in \mathcal F$. In particular, $\nu\left(a_i\textbf{Q}^{\lambda_i}\right)=\nu_Q\left(a_i\textbf{Q}^{\lambda_i}\right)$ for every $i$, $1\leq i\leq n$. Then
\[
\beta\leq \min_{1\leq i\leq n}\left\{\nu\left(a_i\textbf Q^{\lambda_i}\right)\right\}=\min_{1\leq i\leq n}\left\{\nu_Q\left(a_i\textbf Q^{\lambda_i}\right)\right\}\leq \nu_Q(p)\leq \nu(p)=\beta.
\]
Therefore, $\nu_Q(p)=\nu(p)$ and this concludes our proof.
\end{proof}

\begin{Teo}[Theorem 1.1 of \cite{f}]\label{Theoremexistencecompleteseqkpol}
Let $\nu$ be a valuation on $K[x]$. Then there exists a set $\textbf{Q}\subseteq K[x]$ of key polynomials, well-ordered (with the order $Q<Q'$ if and only if $\epsilon(Q)<\epsilon(Q')$), such that $\textbf{Q}$ is complete set for $\nu$.
\end{Teo}
\begin{Obs}
The definition of \textit{complete set of key polynomials} presented in \cite{f} does not require that the degree of the polynomial $Q$ for which $\nu_Q(p)=\nu(p)$ is smaller or equal than $\deg(p)$. This assumption is important and we use this opportunity to fix the definition presented there. The proof presented in \cite{f}, guarantees that this additional property is satisfied, hence the theorem above is still valid.
\end{Obs}

The relation between Spivakovsky's key polynomial and MacLane-Vaqui\'e is given by the following.

\begin{Teo}[Theorem 23 of \cite{wsj}]
Let $Q$ be a Spivakovsky's key polynomial for $\nu$. Then $Q$ is a MacLane-Vaqui\'e key polynomial for $\nu_Q$.
\end{Teo}
We also have the following.
\begin{Teo}[Theorem 26 of \cite{wsj}] Let $Q$ and $Q'$ be two Spivakovsky's key polynomials for $\nu$ such that $Q'\in\Psi(Q)$. Then $Q'$ is a MacLane-Vaqui\'e key polynomial for $\nu_Q$.
\end{Teo}

As for the converse, we have:

\begin{Teo}[Corollary 29 of \cite{wsj}] Let $Q$ be a MacLane-Vaqui\'e key polynomial for $\nu$ and $\nu'$ a valuation of $K[x]$ for which $\nu'(Q)>\nu(Q)$ and $\nu'(f)=\nu(f)$ for every $f\in K[x]$ with $\deg(f)<\deg(Q)$. The $Q$ is a Spivakovsky's key polynomial for $\nu'$.
\end{Teo}

\section{Pseudo-convergent sequences}\label{pseudoconsequ}

Let $\{a_\rho\}_{\rho<\lambda}$ be a pseudo-convergent sequence for $(K,\nu)$. For every polynomial $f(x)\in K[x]$, there exists $\rho_f<\lambda$ such that either
\begin{equation}\label{condforpscstotra}
\nu(f(a_\sigma))=\nu(f(a_{\rho_f}))\mbox{ for every }\rho_f\leq \sigma<\lambda,
\end{equation}
or
\begin{equation}\label{condforpscstoalg}
\nu(f(a_\sigma))>\nu(f(a_{\rho}))\mbox{ for every }\rho_f\leq \rho< \sigma<\lambda.
\end{equation}
\begin{Def}
A pseudo-convergent sequence $\{a_\rho\}_{\rho<\lambda}$ is said to be of \textbf{transcendental type} if for every polynomial $f(x)\in K[x]$ the condition (\ref{condforpscstotra}) holds.
Otherwise, $\{a_\rho\}_{\rho<\lambda}$ is said to be of \textbf{algebraic type}.
\end{Def}

The next two theorems justify the definitions of algebraic and transcendental pseudo-convergent sequences.
\begin{Teo}[Theorem 2 of \cite{Kap}]
If $\{a_\rho\}_{\rho<\lambda}$ is a pseudo-convergent sequence of transcendental type, without a limit in $K$, then there exists an immediate transcendental extension $K(z)$ of $K$ defined by setting $\nu(f(z))$ to be the value $\nu(f(a_{\rho_f}))$ as in condition (\ref{condforpscstotra}). Moreover, for every valuation $\mu$ in some extension $K(u)$ of $K$, if $u$ is a pseudo-limit of $\{a_\rho\}_{\rho<\lambda}$, then there exists a value preserving $K$-isomorphism from $K(u)$ to $K(z)$ taking $u$ to $z$.
\end{Teo}

\begin{Teo}[Theorem 3 of \cite{Kap}]\label{thmonalgimmext}
Let $\{a_\rho\}_{\rho<\lambda}$ be a pseudo-convergent sequence of algebraic type, without a limit in $K$, $q(x)$ a polynomial of smallest degree for which (\ref{condforpscstoalg}) holds and $z$ a root of $q(x)$. Then there exists an immediate algebraic extension of $K$ to $K(z)$ defined as follows: for every polynomial $f(x)\in K[x]$, with $\deg f<\deg q$ we set $\nu(f(z))$ to be the value $\nu(f(a_{\rho_f}))$ as in condition (\ref{condforpscstotra}). Moreover, if $u$ is a root of $q(x)$ and $\mu$ is some extension $K(u)$ of $K$ making $u$ a pseudo-limit of $\{a_\rho\}_{\rho<\lambda}$, then there exists a value preserving $K$-isomorphism from $K(u)$ to $K(z)$ taking $u$ to $z$.
\end{Teo}

\section{Comparison results}\label{compasraccai}
In this section we describe explicitly the relation between key polynomials, pseudo-convegent sequences and minimal pairs.

\begin{Teo}[Theorem 1.2 of \cite{f}]\label{compthemkppsc}
Let $\{a_\rho\}_{\rho<\lambda}\subset K$ be a pseudo-convergent sequence, without a limit in $K$, for which $x$ is a limit. If $\{a_\rho\}_{\rho<\lambda}$ is of transcendental type, then
\[
\textbf{Q}:=\{x-a_\rho\mid \rho<\lambda\}
\]
is a complete set of key polynomials for $\nu$. On the other hand, if $\{a_\rho\}_{\rho<\lambda}$ is of algebraic type, then every polynomial $q(x)$ of minimal degree among the polynomials not fixed by $\{a_\rho\}_{\rho<\lambda}$ is a limit key polynomial for $\nu$.
\end{Teo}

The theorem above gives us a way to interpret pseudo-convergent sequences as key polynomials. The next theorem gives us a way to obtain the opposite relation.

\begin{Prop}[Proposition 1.2 of \cite{g}]\label{propconvofnovaspiv}
Let $\textbf{Q}$ be a complete sequence of key polynomials for $\nu$, without last element. For each $Q\in\textbf{Q}$, let $a_Q\in \overline K$ be a root of $Q$ such that $\mu(x-a_Q)=\delta(Q)$. Then $\{a_Q\}_{Q\in \textbf{Q}}$ is a pseudo-convergent sequence of transcendental type, without a limit in $\overline K$, such that $x$ is a limit for it.
\end{Prop}
We also want to describe the realtion between key polynomials and minimal pairs. The next result gives us such relation.

\begin{Teo}[Theorem 1.1 of \cite{g}]\label{main1}
Let $Q\in K[x]$ be a monic irreducible polynomial and choose a root $a$ of $Q$ such that $\mu(x-a)=\delta(Q)$. Then $Q$ is a key polynomial for $\nu$ if and only if $(a,\delta(Q))$ is a minimal pair for $\nu$. Moreover, $(a,\delta(Q))$ is a minimal pair of definition for $\nu$ if and only if $\nu=\nu_Q$.
\end{Teo}

\section{Examples}\label{Exaampkes}

Let $k$ be a perfect field of characteristic $p>0$ (e.g., $k=\F_p$) and $K=k(y)^{\frac{1}{p^\infty}}$ the perfect hull of $k(y)$. We can consider an embedding
\[
\iota:K\lra k((t^\Q)),
\]
sending $y$ to $t$. Let $\nu_0$ the valuation on $K$ induced by the $t$-adic valuation on $k((t^\Q))$.

Let $x$ be an indeterminate over $K$ and extend $\nu_0$ to $K[x]$ by setting
\[
\nu_1\left(a_0+a_1x+\ldots+a_nx^n\right):=\min_{0\leq i\leq n}\left\{\nu_0(a_i)-\frac{i}{p}\right\}.
\]
In the MacLane-Vaqui\'e's language, we have that $\nu_1$ is a the monomial valuation given by
\[
\nu_1=\left[\nu_0;\nu_1(x)=-\frac{1}{p}\right].
\]
One can show that $\phi_2:=x-y^{-\frac{1}{p}}$ is a key polynomial for $\nu_1$ and we consider the augmented valuation
\[
\nu_2:=\left[\nu_1;\nu_2(\phi_2)=-\frac{1}{p^2}\right].
\]
Then one can show that $\phi_3:=\phi_1-y^{-\frac{1}{p^2}}=x-y^{-\frac 1p}-y^{-\frac{1}{p^2}}$ is a key polynomial for $\nu_2$ and define
\[
\nu_3:=\left[\nu_2;\nu_3(\phi_3)=-\frac{1}{p^3}\right]
\]
We proceed on this manner, until we obtain a sequence of valuations $\{\nu_n\}_{n\in\N}$ for which
\[
\phi_{n+1}=x-\sum_{i=0}^n y^{-\frac{1}{p^i}}\in K[x].
\]
is a key polynomial for $\nu_n$ and
\[
\nu_{n+1}:=\left[\nu_n;\nu_{n+1}(\phi_{n+1})=-\frac{1}{p^{n+1}}\right].
\]
Setting $a_n=\displaystyle\sum_{i=0}^n y^{-\frac{1}{p^i}}$ we have the following.
\begin{Claim}
\begin{itemize}
\item $\{a_n\}_{n\in\N}\subseteq K$ is a pseudo-convergent sequence for $\nu_0$.
\item $x\in K[x]$ is a pseudo-limit for $\{a_n\}_{n\in\N}$, considering the valuation
\[
\nu_\omega:=\displaystyle\lim_{n\rightarrow\infty}\nu_n\mbox{ on }K[x].
\]
\item The pseudo-convergent sequence $\{a_n\}_{n\in \N}$ is of algebraic type.
\item $\phi_\omega:=x^p-x-y^{-1}$ is a monic polynomial, of smallest degree, not fixed by $\{a_n\}_{n\in \N}$.
\end{itemize}
\end{Claim}

\begin{Claim}
The sequence $\{\nu_n\}_{n\in\N}$ is an augmented sequence of valuations on $K[x]$ and $\phi_\omega$ is a limit key polynomial for $\{\nu_n\}_{n\in\N}$.
\end{Claim}
Now take $\gamma\in \Q\cup\{\infty\}$ with $\gamma\geq 0$. Since $0>-\frac{1}{p^n}=\nu_n(\phi_\omega)$ for every $n\in\N$, we can consider the valuation
\begin{equation}\label{equandoasjdpsoa}
\nu_{\omega+1}:=\left[\{\nu_n\}_{n\in\N},\nu_{\omega+1}(\phi_\omega)=\gamma\right].
\end{equation}

\begin{Obs}
Let
\[
\eta:=\sum_{i=0}^\infty y^{-\frac{1}{p^i}}\in\overline{K},
\]
which is a root of $\phi_\omega$. If $\gamma=\infty$, then $\nu_{\omega+1}$ induces a valuation on $K(\eta)=K[x]/(\phi_\omega)$ which is exactly the valuation given on Theorem \ref{thmonalgimmext}. In this case, the pseudo-convergent sequence $\{\nu_n\}_{n\in\N}$ can be thought of as a ``pseudo-convergent sequence of algebraic type with an algebraic limit" (because in this case $\eta$ is a limit for it).
\end{Obs}

The construction of $\nu_{\omega+1}$ above can be generalized in the following way. Let $\eta'\in k((t^\Q))$ and extend $\iota:K\lra k((t^\Q))$ to a map $K[x]\lra k((t^\Q))$ (which we call again $\iota$) sending $x$ to $\eta'$.  Since $k((t^\Q))$ is algebraically closed, this defines a valuation on $K[x]$ by setting
\[
\nu_{\eta'}(a_0+\ldots+a_nx^n):=\sum_{i=1}^r\nu_t(\eta'-\eta_i)\mbox{ where }\iota(a_0)+\ldots+\iota(a_n)x^n=\prod_{i=1}^r(x-\eta_i).
\]
In particular, $\nu_{\eta'}(x-a):=\nu_t(\eta'-\iota(a))$ for each $a\in K$.
\begin{Claim}\label{eqgvalrotspiwser}
The valuation constructed in (\ref{equandoasjdpsoa}) is equals to $\nu_{\eta}$ if $\gamma=\infty$ and to $\nu_{\eta'}$ where
\[
\eta'=\eta+a\gamma+\eta''\mbox{ such that }\SU(\eta'')>\gamma\mbox{ and }a\neq 0,
\]
if $0\leq \gamma<\infty$. Moreover, if $\eta''$ is transcencental over $\iota(K)$, then $\{a_n\}$ is a ``pesudo-convergent sequence of algebraic type with a transcendental pseudo-limit" (because $\eta'$ is a limit of it).
\end{Claim}

So far, we have constucted an example where the sequence of key polynomials of order type $\omega$ ``is not enough to construct the valuation". In terms of pseudo-convergent sequences, this means that the pseudo-convergent sequence is of algebraic type. We will now continue the construction, starting from the valuation $\nu_{\omega+1}$ defined by the limit key polynomial $\phi_\omega$.

Let now $\gamma=0$ and $\phi_{\omega+1}=\phi_\omega+1$. Then $\phi_{\omega+1}$ is a key polynomial for $\nu_{\omega+1}$ and we can define the valuation
\[
\nu_{\omega+2}:=\left[\nu_{\omega+1};\nu_{\omega+2}(\phi_{\omega+1})=\frac{1}{p}\right].
\]
One can prove that
\[
\phi_{\omega+2}:=\phi_{\omega+1}-y^{\frac{1}{p}}=\phi_\omega-1-y^{\frac{1}{p}}
\]
is a key polynomial for $\nu_{\omega+2}$. We set
\[
\nu_{\omega+3}:=\left[\nu_{\omega+2};\nu_{\omega+3}(\phi_{\omega+2})=\frac{1+p}{p}\right].
\]
We can construct a sequence of valuations $\{\nu_{\omega+n}\}_{n\in\N}$ such that
\[
\phi_{\omega+n+1}=\phi_\omega-1-\sum_{i=1}^n y^{\frac{1+\ldots+p^i}{p^i}}\in K[x].
\]
is a key polynomial for $\nu_{\omega+n}$ and
\[
\nu_{\omega+n+1}:=\left[\nu_{\omega+n};\nu_{\omega+n+1}(\phi_{\omega+n})=\frac{1+\ldots+p^n}{p^n}\right].
\]
\begin{Claim}
The sequence $\{\nu_{\omega+n}\}_{n\in\N}$ is an augmented sequence of valuations and $\phi_{2\omega}:=\phi_\omega^p-y\phi_\omega-1$ is a limit key polynomial for $\{\nu_{\omega+n}\}_{n\in\N}$.
\end{Claim}

If $\gamma'$ is such that $\gamma>\frac{p}{p-1}$, then we can define the valuation
\begin{equation}\label{equandoasjdpsoaw}
\nu_{2\omega+1}:=\left[\{\nu_{\omega+n}\}_{n\in\N},\nu_{2\omega+1}(\phi_{2\omega})=\gamma'\right].
\end{equation}
One can deduce from what was said before, that for every $i$, $1\leq i\leq 2\omega$ the polynomial $\phi_{i}$ is a Spivakovsky's key polynomial for $\nu_{2\omega+1}$ and that the truncation $\nu_{\phi_i}$ is equal to $\nu_i$.

\noindent{\footnotesize JOSNEI NOVACOSKI\\
Departamento de Matem\'atica--UFSCar\\
Rodovia Washington Lu\'is, 235\\
13565-905 - S\~ao Carlos - SP\\
Email: {\tt josnei@dm.ufscar.br} \\\\
\end{document}